\documentclass[11pt]{article}

\usepackage[margin=1.08in]{geometry}
\usepackage{amsmath,amssymb,amsthm,mathtools}
\usepackage{microtype}
\usepackage[hidelinks]{hyperref}
\hypersetup{pdftitle={Joint Continuity and Selberg--ODE Equivalence for the Sine-beta Pair Correlation Function},pdfauthor={Weiyang Fang}}
\usepackage{enumitem}

\newtheorem{theorem}{Theorem}[section]
\newtheorem{lemma}[theorem]{Lemma}
\newtheorem{proposition}[theorem]{Proposition}
\newtheorem{corollary}[theorem]{Corollary}
\theoremstyle{remark}
\newtheorem{remark}[theorem]{Remark}

\newcommand{\R}{\mathbb R}
\newcommand{\C}{\mathbb C}
\newcommand{\E}{\mathbb E}
\newcommand{\Law}{\mathcal L}
\newcommand{\Sine}{\operatorname{Sine}}

\newcommand{\ii}{\mathrm i}
\newcommand{\ee}{\mathrm e}
\newcommand{\Id}{I}
\newcommand{\rhoS}{\rho^{(2)}}
\newcommand{\Dop}{\mathsf D}

\title{Joint Continuity and Selberg--ODE Equivalence for the\\
$\Sine_\beta$ Pair Correlation Function}
\author{Weiyang Fang}
\date{}

\begin{document}
\maketitle

\begin{abstract}
We address two questions posed by Qu and Valk\'o in their study of the pair correlation function of the $\Sine_\beta$ process.  First, we prove that the pair correlation function admits a jointly continuous version in the inverse-temperature and spatial parameters on $(0,\infty)\times\R$, removing the restriction $\beta>2$ in their joint-continuity result.  The argument uses a general observation: separate weak continuity of a family of probability laws, together with stochastic monotonicity in one parameter, implies joint weak continuity.  Applied to the terminal value of the Qu--Valk\'o diffusion, their Palm-density formula then gives the result without differentiating the Fourier expansion.

Second, for $\beta=2n$ we give a direct proof that the Qu--Valk\'o matrix-recursion power series agrees with the Selberg-integral representation recorded by Forrester.  A Krawtchouk transform converts the Qu--Valk\'o system into a $(2n+1)$-dimensional differential system with parameter $-(n+1)$, while a centered Aomoto--Selberg trace system has parameter $n+1$.  An explicit triangular differential intertwiner reflects the parameter $\eta\mapsto-\eta$.  Matching the unique Frobenius branch of exponent $2$ gives exactly the Selberg normalization.  As a consequence, every coefficient of the Qu--Valk\'o recursion is identified with an even centered moment of the corresponding Jacobi--Selberg trace statistic.
\end{abstract}

\medskip
\noindent\textbf{Keywords.} $\Sine_\beta$ process; pair correlation; beta ensembles; Selberg integral; stochastic order; Krawtchouk transform.\\
\noindent\textbf{2020 Mathematics Subject Classification.} 60B20, 60G55, 33C45.

\section{Introduction}

The $\Sine_\beta$ process is the translation-invariant bulk scaling limit of beta ensembles.  For the classical values $\beta=1,2,4$, its correlation functions inherit determinantal or Pfaffian descriptions.  For general $\beta>0$, the process is instead naturally described through stochastic differential equations and random operators.

Qu and Valk\'o recently obtained an SDE representation of the pair correlation function $\rhoS_\beta$ for all $\beta>0$ and developed several consequences of it~\cite{QuValko}.  In particular, they proved separate continuity in $\beta$ and in the spatial variable and joint continuity when $\beta>2$ away from the diagonal.  They asked whether joint continuity holds for every $\beta>0$~\cite[Problem~4]{QuValko}.  For even inverse temperature $\beta=2n$, they also derived a finite-dimensional ODE and a convergent power series for $\rhoS_{2n}$.  Forrester's earlier Selberg-integral formulas~\cite{Forrester1992,Forrester1994,ForresterBook} lead to another representation of the same limiting pair correlation, and Qu--Valk\'o asked for a direct proof that the two representations agree~\cite[Problem~6]{QuValko}.

A complementary line of work concerns fusion asymptotics, in which several arguments of a correlation function approach one another.  In~\cite{FangSecond}, the author computes the first normalized correction for $m$ merging points in the supercritical regime $m\beta>1$.  The critical and subcritical regimes $m\beta\le1$, with logarithmic and fractional-power corrections, are treated in~\cite{FangCritical}.  These short-distance results provide context for the coefficientwise comparison in Section~\ref{sec:coeff}; the present proofs of joint continuity and exact Selberg--ODE equivalence do not use the fusion expansions.

The two main results of this paper answer these questions.

\begin{theorem}[Joint continuity]\label{thm:intro-cont}
The Qu--Valk\'o representation of $\rhoS_\beta(0,\lambda)$ for $\lambda\ne0$ extends to a jointly continuous function on
\[
 (0,\infty)\times\R.
\]
The continuous extension satisfies
\[
 \rhoS_\beta(0,0)=0,\qquad \beta>0.
\]
\end{theorem}

The proof does not use estimates for $\partial_\lambda\rhoS_\beta$.  The terminal law of the Qu--Valk\'o diffusion is weakly continuous in each parameter separately, and the diffusion is almost surely increasing in the spatial parameter.  A monotone squeeze at the level of probability laws upgrades separate weak continuity to joint weak continuity.  The pair correlation is then the expectation of a bounded periodic observable that varies uniformly with $\beta$ on compact parameter sets.

Our second result concerns the even-beta formulas.  Put $N=2n$ and
\begin{equation}\label{eq:intro-Dn}
 D_n:=\frac{n^{2n}(n!)^3}{(2n)!(3n)!}.
\end{equation}
Let
\begin{equation}\label{eq:intro-weight}
 W_n(u):=\prod_{j=1}^{N}u_j^{-1+1/n}(1-u_j)^{-1+1/n}
 \prod_{1\le j<k\le N}|u_j-u_k|^{2/n},
 \qquad u\in[0,1]^N,
\end{equation}
and let
\[
 \mathcal S_n:=\int_{[0,1]^N}W_n(u)\,du.
\]
Define the centered trace transform
\begin{equation}\label{eq:intro-Gn}
 G_n(\lambda):=
 \frac{\ee^{-\ii n\lambda}}{\mathcal S_n}
 \int_{[0,1]^N}\ee^{\ii\lambda\sum_{j=1}^{N}u_j}W_n(u)\,du.
\end{equation}
Forrester's representation, in the normalization used by Qu--Valk\'o, is
\begin{equation}\label{eq:intro-forrester}
 4\pi^2\rhoS_{2n}(0,\lambda)=D_n\lambda^{2n}G_n(\lambda).
\end{equation}
On the other hand, Qu--Valk\'o define vectors $s_m\in\C^n$ by the matrix recursion recalled in Section~\ref{sec:even-prelim}, and obtain
\begin{equation}\label{eq:intro-qvseries}
 4\pi^2\rhoS_{2n}(0,\lambda)
 =2\sum_{j\ge1}v_n^Ts_{2j}\lambda^{2j}.
\end{equation}

\begin{theorem}[Direct equivalence]\label{thm:intro-equiv}
For every $n\ge1$ and every $\lambda\in\C$,
\begin{equation}\label{eq:intro-equivalence}
 2\sum_{j\ge1}v_n^Ts_{2j}\lambda^{2j}
 =D_n\lambda^{2n}G_n(\lambda).
\end{equation}
Consequently, the Qu--Valk\'o power series and the Forrester Selberg integral are identical as entire functions.
\end{theorem}

The proof is finite-dimensional.  We transform both descriptions into members of one family of first-order systems
\begin{equation}\label{eq:intro-family}
 \Dop K=(Y^{(\eta)}-\ii\lambda X)K,
 \qquad \Dop:=\lambda\frac{d}{d\lambda},
\end{equation}
with explicit $(2n+1)\times(2n+1)$ matrices $X$ and $Y^{(\eta)}$.  The Qu--Valk\'o side gives $\eta=-(n+1)$ after a Krawtchouk transform, while the centered Selberg trace gives $\eta=n+1$.  An explicit triangular transform intertwines the systems at $\eta$ and $-\eta$.  A one-dimensional Frobenius matching then determines the proportionality constant, which is exactly $D_n$.

Section~\ref{sec:prelim} recalls the ingredients from~\cite{QuValko}.  Section~\ref{sec:continuity} proves Theorem~\ref{thm:intro-cont}.  Sections~\ref{sec:equiv-start}--\ref{sec:frob} prove Theorem~\ref{thm:intro-equiv}; the Selberg differential system is derived directly by integration by parts.  Section~\ref{sec:coeff} records a coefficientwise consequence.

\section{Preliminaries from the Qu--Valk\'o representation}\label{sec:prelim}

\subsection{The diffusion and the Palm density}

For $\beta>0$ and $\lambda\in\R$, let $\alpha_{\lambda,\beta}$ denote the strong solution used in~\cite{QuValko} after specializing the Jacobi parameter to $\delta=\beta/2$:
\begin{equation}\label{eq:sde}
 d\alpha_{\lambda,\beta}(u)
 =\lambda\frac{\beta}{4}\ee^{\beta u/4}\,du
 -\frac{\beta}{2}\sin\bigl(\alpha_{\lambda,\beta}(u)\bigr)\,du
 +\Re\!\left[(\ee^{-\ii\alpha_{\lambda,\beta}(u)}-1)dZ(u)\right],
\end{equation}
with entrance condition
\begin{equation}\label{eq:entrance}
 \lim_{u\to-\infty}\alpha_{\lambda,\beta}(u)=0.
\end{equation}
Set
\begin{equation}\label{eq:law-def}
 A_{\beta,\lambda}:=\alpha_{\lambda,\beta}(0),
 \qquad
 \mu_{\beta,\lambda}:=\Law(A_{\beta,\lambda}).
\end{equation}
Qu and Valk\'o prove that for fixed $\beta$ the map $\lambda\mapsto A_{\beta,\lambda}$ can be realized almost surely as an analytic, strictly increasing function, and they construct a common realization in which $\beta\mapsto A_{\beta,\lambda}$ is almost surely continuous for fixed $\lambda$; see~\cite[Propositions~6 and~21]{QuValko}.

For $\delta>0$, introduce the $2\pi$-periodic density
\begin{equation}\label{eq:hdelta}
 h_\delta(x)
 =\frac{1}{2\pi}\frac{\Gamma(1+\delta)^2}{\Gamma(1+2\delta)}
 |1-\ee^{\ii x}|^{2\delta}.
\end{equation}
Its absolutely convergent Fourier expansion is
\begin{equation}\label{eq:hfourier}
 h_\delta(x)
 =\frac1{2\pi}+\frac1\pi\sum_{k\ge1}
 \frac{(-\delta)^{\uparrow k}}{(1+\delta)^{\uparrow k}}\cos(kx),
\end{equation}
where $a^{\uparrow k}=a(a+1)\cdots(a+k-1)$.  The Palm identification in~\cite{QuValko} implies
\begin{equation}\label{eq:palm-expectation}
 \rhoS_\beta(0,\lambda)
 =\frac1{2\pi}\E\bigl[h_{\beta/2}(A_{\beta,|\lambda|})\bigr],
 \qquad \lambda\ne0.
\end{equation}
Equivalently, with
\begin{equation}\label{eq:Gbeta}
 \mathcal G_\beta(x)
 :=\frac{\Gamma(1+\beta/2)^2}{4\pi^2\Gamma(1+\beta)}
 |1-\ee^{\ii x}|^\beta,
\end{equation}
we have
\begin{equation}\label{eq:palm-G}
 \rhoS_\beta(0,\lambda)
 =\int_\R \mathcal G_\beta(x)\,\mu_{\beta,|\lambda|}(dx).
\end{equation}

\subsection{The even-beta ODE and power series}\label{sec:even-prelim}

Fix $n\ge1$ and set $\beta=2n$.  Let $A_n,B_n\in\R^{n\times n}$ be given by
\begin{align}
 [A_n]_{k,k}&=-k^2,&
 [A_n]_{k,k-1}&=\frac12k(k+n),&
 [A_n]_{k,k+1}&=\frac12k(k-n),\label{eq:AnBn}\
 [B_n]_{k,k}&=k,
\end{align}
with out-of-range entries omitted.  Let
\begin{equation}\label{eq:envn}
 e_n=(1,0,\ldots,0)^T,
 \qquad
 [v_n]_k=(-1)^k\frac{\binom{2n}{n+k}}{\binom{2n}{n}},
 \quad 1\le k\le n.
\end{equation}
Qu and Valk\'o define
\begin{equation}\label{eq:sm-rec}
 s_0=-\frac{n+1}{2}A_n^{-1}e_n,
 \qquad
 s_m=\ii\left(m\Id-\frac2nA_n\right)^{-1}B_ns_{m-1},
 \quad m\ge1,
\end{equation}
and prove the entire power-series representation
\begin{equation}\label{eq:qv-series}
 4\pi^2\rhoS_{2n}(0,\lambda)
 =2\sum_{j\ge1}v_n^Ts_{2j}\lambda^{2j}.
\end{equation}
They also show that if
\[
 q_k(\lambda):=\E\ee^{\ii k\alpha_{\lambda,2n}(0)},
 \qquad q=(q_1,\ldots,q_n)^T,
\]
then
\begin{equation}\label{eq:q-ode-matrix}
 \frac n2\lambda q'(\lambda)
 =\left(\ii\frac n2\lambda B_n+A_n\right)q(\lambda)
 +\frac{n+1}{2}e_n,
 \qquad q(0)=\mathbf 1.
\end{equation}

Forrester's Selberg representation, in the form recorded in~\cite{QuValko}, is exactly~\eqref{eq:intro-forrester}, with $G_n$ given by~\eqref{eq:intro-Gn}.

\section{Joint continuity for all \texorpdfstring{$\beta>0$}{beta > 0}}\label{sec:continuity}

We first isolate the elementary topological input.

\begin{lemma}[Separate weak continuity and stochastic monotonicity]\label{lem:monotone}
Let $I,J\subset\R$ be intervals and let $\{\nu_{s,t}:(s,t)\in I\times J\}$ be probability measures on $\R$.  Assume:
\begin{enumerate}[label=(\roman*)]
\item for each fixed $t$, $s\mapsto\nu_{s,t}$ is weakly continuous;
\item for each fixed $s$, $t\mapsto\nu_{s,t}$ is weakly continuous;
\item for each fixed $s$ and $t_1\le t_2$,
\[
 \nu_{s,t_1}\le_{\rm st}\nu_{s,t_2}.
\]
\end{enumerate}
Then $(s,t)\mapsto\nu_{s,t}$ is jointly weakly continuous.
\end{lemma}

\begin{proof}
Fix $(s,t)$ and $(s_m,t_m)\to(s,t)$.  If $f$ is bounded, continuous and nondecreasing, set
\[
 H_f(a,b):=\int f\,d\nu_{a,b}.
\]
The function $H_f$ is separately continuous and nondecreasing in its second variable.  If $t$ is an interior point of $J$, then for every sufficiently small $\varepsilon>0$ and all large $m$,
\[
 H_f(s_m,t-\varepsilon)\le H_f(s_m,t_m)\le H_f(s_m,t+\varepsilon).
\]
First let $m\to\infty$ and then $\varepsilon\downarrow0$.  Separate continuity gives
\begin{equation}\label{eq:increasing-test}
 \int f\,d\nu_{s_m,t_m}\longrightarrow\int f\,d\nu_{s,t}.
\end{equation}
At an endpoint of $J$ the same argument is one-sided.

It remains to note that bounded continuous nondecreasing functions are convergence determining on $\R$.  Indeed, if $x$ is a continuity point of the distribution function of $\nu_{s,t}$, approximate $\mathbf1_{(x,\infty)}$ from below and above by continuous nondecreasing functions whose transition intervals shrink to $x$.  Equation~\eqref{eq:increasing-test} then gives convergence of the tail probabilities at every such $x$, hence weak convergence.
\end{proof}

\begin{proposition}\label{prop:joint-law}
The map
\[
 (\beta,\lambda)\longmapsto\mu_{\beta,\lambda}
\]
is jointly weakly continuous on $(0,\infty)\times\R$.
\end{proposition}

\begin{proof}
For fixed $\lambda$, the almost sure continuity in $\beta$ from~\cite[Proposition~21]{QuValko} implies weak continuity of $\beta\mapsto\mu_{\beta,\lambda}$.  For fixed $\beta$, the almost sure analyticity in $\lambda$ implies weak continuity in $\lambda$.  Finally, the almost sure monotonicity in~\cite[Proposition~6]{QuValko} gives
\[
 \lambda_1\le\lambda_2
 \quad\Longrightarrow\quad
 \mu_{\beta,\lambda_1}\le_{\rm st}\mu_{\beta,\lambda_2}.
\]
Lemma~\ref{lem:monotone} applies.
\end{proof}

\begin{proof}[Proof of Theorem~\ref{thm:intro-cont}]
Let $(\beta_m,\lambda_m)\to(\beta,\lambda)$ with $\beta>0$.  Proposition~\ref{prop:joint-law} gives
\[
 \mu_{\beta_m,|\lambda_m|}\Rightarrow\mu_{\beta,|\lambda|}.
\]
Choose $0<b<\beta<B<\infty$ so that $\beta_m\in[b,B]$ for all sufficiently large $m$.  The function
\[
 (\gamma,x)\longmapsto\mathcal G_\gamma(x)
\]
is continuous on $[b,B]\times[0,2\pi]$.  At the zeros of $|1-\ee^{\ii x}|$, the positivity of the lower exponent bound $b$ gives continuity uniformly in $\gamma\in[b,B]$.  By periodicity,
\begin{equation}\label{eq:G-uniform}
 \|\mathcal G_{\beta_m}-\mathcal G_\beta\|_\infty\longrightarrow0.
\end{equation}
Since $\mathcal G_\beta$ is bounded and continuous,
\[
 \int\mathcal G_\beta\,d\mu_{\beta_m,|\lambda_m|}
 \longrightarrow
 \int\mathcal G_\beta\,d\mu_{\beta,|\lambda|}.
\]
Together with~\eqref{eq:G-uniform} and~\eqref{eq:palm-G}, this proves joint continuity whenever $\lambda\ne0$.

For $\lambda=0$, the solution of~\eqref{eq:sde} is identically zero, so $\mu_{\beta,0}=\delta_0$.  Because $\mathcal G_\beta(0)=0$, the same argument yields
\[
 \lim_{(\gamma,t)\to(\beta,0)}\rhoS_\gamma(0,t)=0.
\]
Thus the representation extends jointly continuously with value $0$ on the diagonal.
\end{proof}

\begin{corollary}\label{cor:uniform-small}
For every $0<b<B<\infty$,
\[
 \lim_{\varepsilon\downarrow0}
 \sup_{\substack{b\le\beta\le B\\|\lambda|\le\varepsilon}}
 \rhoS_\beta(0,\lambda)=0.
\]
\end{corollary}

\begin{proof}
Theorem~\ref{thm:intro-cont} gives uniform continuity on the compact set $[b,B]\times[-1,1]$, and the continuous version vanishes on $[b,B]\times\{0\}$.
\end{proof}

\begin{remark}
The restriction $\beta>2$ in the derivative-based argument of~\cite{QuValko} enters through summability of bounds for the derivatives of the Fourier moments.  Proposition~\ref{prop:joint-law} uses monotonicity before the observable $\mathcal G_\beta$ is applied, so no differentiated Fourier series is needed.
\end{remark}

\section{A common differential system for the even-beta formulas}\label{sec:equiv-start}

From now on fix $n\ge1$ and put $N=2n$.  For $\eta\in\C$, define matrices $X,Y^{(\eta)}\in\C^{(N+1)\times(N+1)}$, indexed by $r=0,\ldots,N$, by
\begin{align}
 X_{r,r-1}&=-\frac r4,
 &X_{r,r+1}&=-(N-r),\label{eq:Xdef}\\
 Y^{(\eta)}_{r,r}&=-\frac r n(\eta+3n-r),
 &Y^{(\eta)}_{r,r-2}&=-\frac{r(r-1)}{4n},\label{eq:Ydef}
\end{align}
with all out-of-range entries interpreted as zero.  We write
\[
 \Dop:=\lambda\frac d{d\lambda}.
\]
Our common system is
\begin{equation}\label{eq:eta-system}
 \Dop K=(Y^{(\eta)}-\ii\lambda X)K.
\end{equation}

\subsection{The Qu--Valk\'o side and a Krawtchouk transform}

Writing~\eqref{eq:q-ode-matrix} componentwise and absorbing the inhomogeneous term by $q_0\equiv1$ gives
\begin{equation}\label{eq:q-positive}
 \Dop q_k
 =\frac{k(k+n)}nq_{k-1}-\frac{2k^2}{n}q_k
 +\frac{k(k-n)}nq_{k+1}+\ii k\lambda q_k,
 \qquad 1\le k\le n.
\end{equation}
Define
\[
 q_{-k}(\lambda):=q_k(-\lambda),\qquad 1\le k\le n.
\]
Then~\eqref{eq:q-positive} extends to the closed system
\begin{equation}\label{eq:q-full}
 \Dop q_k
 =\frac{k(k+n)}nq_{k-1}-\frac{2k^2}{n}q_k
 +\frac{k(k-n)}nq_{k+1}+\ii k\lambda q_k,
 \qquad -n\le k\le n,
\end{equation}
where the boundary terms at $k=\pm n$ are absent.

Set
\begin{equation}\label{eq:wk}
 w_k:=(-1)^k\frac{\binom{N}{n+k}}{\binom{N}{n}},
 \qquad -n\le k\le n.
\end{equation}
For real $\lambda$, $q_{-k}(\lambda)=\overline{q_k(\lambda)}$, and the Qu--Valk\'o finite Fourier sum becomes
\begin{equation}\label{eq:Fq}
 F_n(\lambda):=4\pi^2\rhoS_{2n}(0,\lambda)
 =\sum_{k=-n}^{n}w_kq_k(\lambda).
\end{equation}
Both sides extend to the same entire even function.

Introduce polynomials $P_r(k)$, $0\le r\le N$, by
\begin{equation}\label{eq:kraw-gen}
 (1+t/2)^{n+k}(1-t/2)^{n-k}
 =\sum_{r=0}^{N}\binom NrP_r(k)t^r.
\end{equation}
Comparison of coefficients gives
\begin{equation}\label{eq:kraw-rec}
 (N-r)P_{r+1}(k)=kP_r(k)-\frac r4P_{r-1}(k).
\end{equation}
Let $S$ be the $(N+1)\times(N+1)$ matrix, with columns indexed by $k=-n,\ldots,n$, defined by
\begin{equation}\label{eq:Sdef}
 S_{r,k}:=w_kP_r(k).
\end{equation}
Since $P_r$ has degree $r$, the matrix $S$ is invertible.

Let $Q=(q_{-n},\ldots,q_n)^T$, let $K_0=\operatorname{diag}(-n,\ldots,n)$, and let $C$ denote the constant tridiagonal part of~\eqref{eq:q-full}, so that
\begin{equation}\label{eq:Qmatrix}
 \Dop Q=(C+\ii\lambda K_0)Q.
\end{equation}

\begin{lemma}[Krawtchouk conjugation]\label{lem:kraw}
The matrix $S$ satisfies
\begin{equation}\label{eq:kraw-conj}
 SK_0=-XS,
 \qquad
 SC=(Y^{(-(n+1))}-2\Id)S.
\end{equation}
Moreover,
\begin{equation}\label{eq:S-one}
 S\mathbf1=\frac{(-1)^n}{\binom{2n}{n}}e_N,
\end{equation}
where $e_N$ is the last coordinate vector of $\C^{N+1}$.
\end{lemma}

\begin{proof}
The first identity in~\eqref{eq:kraw-conj} is exactly~\eqref{eq:kraw-rec}.  For the second, the ratios
\[
 \frac{w_{k+1}}{w_k}=-\frac{n-k}{n+k+1},
 \qquad
 \frac{w_{k-1}}{w_k}=-\frac{n+k}{n-k+1}
\]
give
\begin{align}
 \frac{(SC)_{r,k}}{w_k}
 =\frac1n\Big(& (n+k)(k-1)P_r(k-1)-2k^2P_r(k)\notag\\
 &-(n-k)(k+1)P_r(k+1)\Big).\label{eq:kraw-diff}
\end{align}
The generating function~\eqref{eq:kraw-gen} yields
\begin{align}
 &(n+k)(k-1)P_r(k-1)-2k^2P_r(k)-(n-k)(k+1)P_r(k+1)\notag\\
 &\qquad=\bigl(r(r+1-N)-N\bigr)P_r(k)-\frac{r(r-1)}4P_{r-2}(k).\label{eq:kraw-identity}
\end{align}
Indeed, after multiplying by $\binom Nr t^r$ and summing in $r$, the right-hand side is obtained by applying
\[
 \theta^2+(1-N)\theta-N-\frac{t^2}{4}(N-\theta)(N-\theta-1),
 \qquad \theta=t\partial_t,
\]
to~\eqref{eq:kraw-gen}; direct differentiation produces the generating function of the left-hand side.  Since
\[
 Y^{(-(n+1))}_{r,r}=\frac{r(r+1-N)}n,
 \qquad
 Y^{(-(n+1))}_{r,r-2}=-\frac{r(r-1)}{4n},
\]
identity~\eqref{eq:kraw-identity} proves the second relation in~\eqref{eq:kraw-conj}.

Finally, summing~\eqref{eq:kraw-gen} against $w_k$ gives
\begin{align*}
 \sum_{r=0}^{N}\binom Nr(S\mathbf1)_rt^r
 &=\frac{(-1)^n}{\binom Nn}
 \sum_{j=0}^{N}(-1)^j\binom Nj(1+t/2)^j(1-t/2)^{N-j}\\
 &=\frac{(-1)^n}{\binom Nn}t^N,
\end{align*}
which is~\eqref{eq:S-one}.
\end{proof}

Set
\begin{equation}\label{eq:Mdef}
 L:=SQ,
 \qquad
 M:=\lambda^2L.
\end{equation}
Lemma~\ref{lem:kraw} gives
\begin{equation}\label{eq:Msystem}
 \Dop M=(Y^{(-(n+1))}-\ii\lambda X)M.
\end{equation}
The first row of $S$ is $(w_k)_{k=-n}^{n}$, hence
\begin{equation}\label{eq:Mfirst}
 M_0(\lambda)=\lambda^2F_n(\lambda).
\end{equation}
Since $Q(0)=\mathbf1$, we also have
\begin{equation}\label{eq:Masymp}
 M(\lambda)=c_Q\lambda^2e_N+O(\lambda^3),
 \qquad
 c_Q:=\frac{(-1)^n}{\binom{2n}{n}}.
\end{equation}

\section{The centered Selberg trace system}\label{sec:selberg-system}

We next show directly that the Selberg integral belongs to the same family~\eqref{eq:eta-system}, now with parameter $n+1$.  The recurrence is a specialization of the Aomoto--Selberg differential system; related trace recurrences are developed in~\cite{Aomoto,ForresterKumar}.  We include the derivation to make the parameter normalization explicit.

For $p=0,\ldots,N$, define
\begin{equation}\label{eq:Hdef}
 \widehat H_p(x):=
 \frac{1}{\binom Np\,\mathcal S_n}
 \int_{[0,1]^N}W_n(u)\ee^{-x\sum_j u_j}
 e_p(1-u_1,\ldots,1-u_N)\,du,
\end{equation}
where $e_p$ is the elementary symmetric polynomial.  Put $\widehat H_{-1}=\widehat H_{N+1}=0$.

\begin{lemma}[Jacobi trace differential recurrence]\label{lem:Hrec}
For $0\le p\le N$,
\begin{equation}\label{eq:Hrec}
 (N-p)x\widehat H_{p+1}
 =\bigl((N-p)x+B_p\bigr)\widehat H_p
 +x\widehat H_p'-D_p\widehat H_{p-1},
\end{equation}
where
\begin{equation}\label{eq:BD}
 B_p=\frac{p(4n+1-p)}n,
 \qquad
 D_p=\frac{p(2n-p+1)}n.
\end{equation}
\end{lemma}

\begin{proof}
We include the integration-by-parts argument to fix the normalization.  Set $y_i=1-u_i$.  The weight is unchanged in form,
\[
 \prod_i y_i^a(1-y_i)^a\prod_{i<j}|y_i-y_j|^{\vartheta},
 \qquad
 a=-1+\frac1n,
 \quad
 \vartheta=\frac2n,
\]
while $\ee^{-x\sum u_i}=\ee^{-Nx}\ee^{x\sum y_i}$.  For $p\ge1$, integrate
\[
 \sum_{i=1}^N\partial_{y_i}
 \left[y_i(1-y_i)e_{p-1}^{(i)}(y)\ee^{x\sum_jy_j}
 \prod_\ell y_\ell^a(1-y_\ell)^a
 \prod_{j<k}|y_j-y_k|^\vartheta\right],
\]
where $e_{p-1}^{(i)}$ omits $y_i$.  Here $a+1=1/n>0$, so the boundary terms at $0$ and $1$ vanish after multiplication by $y_i(1-y_i)$.  The pair-collision singularities created by differentiating the Vandermonde factor are locally integrable because $\vartheta=2/n>0$, so the integration by parts is legitimate (or may be justified first off the collision hyperplanes and then by a limiting argument).

The elementary-symmetric identities
\begin{align*}
 \sum_i e_{p-1}^{(i)}&=(N-p+1)e_{p-1},\\
 \sum_i y_i e_{p-1}^{(i)}&=p e_p,\\
 \sum_i y_i^2 e_{p-1}^{(i)}&=e_1e_p-(p+1)e_{p+1}
\end{align*}
and the pairwise symmetrization
\begin{align*}
 &\sum_i y_i(1-y_i)e_{p-1}^{(i)}\sum_{j\ne i}\frac1{y_i-y_j}\\
 &\qquad=\binom{N-p+1}{2}e_{p-1}
 -\frac{p(2N-p-1)}2e_p
\end{align*}
give, before the gauge factor $\ee^{-Nx}$ is restored,
\begin{align*}
 (N-p)x\widetilde H_{p+1}
 ={}&x\widetilde H_p'-px\widetilde H_p
 +p\left(2+2a+\frac\vartheta2(2N-p-1)\right)\widetilde H_p\\
 &-p\left(1+a+\frac\vartheta2(N-p)\right)\widetilde H_{p-1},
\end{align*}
where $\widetilde H_p=\ee^{Nx}\widehat H_p$.  Substitution of $N=2n$, $a=-1+1/n$ and $\vartheta=2/n$ gives~\eqref{eq:Hrec}--\eqref{eq:BD}.  The case $p=0$ follows directly by differentiating $\widehat H_0$.
\end{proof}

Write $\widehat H=(\widehat H_0,\ldots,\widehat H_N)^T$.  Lemma~\ref{lem:Hrec} is equivalent to
\begin{equation}\label{eq:Hmatrix}
 x\frac d{dx}\widehat H=(xX_0+Y_0)\widehat H,
\end{equation}
where
\begin{align}
 (X_0)_{p,p}&=-(N-p),&(X_0)_{p,p+1}&=N-p,\label{eq:X0}\\
 (Y_0)_{p,p}&=-\frac{p(4n+1-p)}n,&
 (Y_0)_{p,p-1}&=\frac{p(2n-p+1)}n.\label{eq:Y0}
\end{align}

Define the lower-triangular matrix
\begin{equation}\label{eq:Tdef}
 T_{r,p}:=(-1)^p2^{p-r}\binom rp,
 \qquad 0\le p\le r\le N.
\end{equation}
The translation formula for elementary symmetric polynomials gives
\begin{equation}\label{eq:translation}
 \frac{e_r(u_1-1/2,\ldots,u_N-1/2)}{\binom Nr}
 =\sum_{p=0}^{r}T_{r,p}
 \frac{e_p(1-u_1,\ldots,1-u_N)}{\binom Np}.
\end{equation}
Set
\begin{equation}\label{eq:Kselberg}
 K_r(\lambda):=
 \ee^{-\ii n\lambda}\sum_{p=0}^{r}T_{r,p}\widehat H_p(-\ii\lambda),
 \qquad 0\le r\le N.
\end{equation}
Then, in particular,
\begin{equation}\label{eq:K0G}
 K_0(\lambda)=G_n(\lambda).
\end{equation}

\begin{lemma}[Centered Selberg system]\label{lem:centered}
The vector $K=(K_0,\ldots,K_N)^T$ satisfies
\begin{equation}\label{eq:Kplus}
 \Dop K=(Y^{(n+1)}-\ii\lambda X)K.
\end{equation}
\end{lemma}

\begin{proof}
The binomial matrix~\eqref{eq:Tdef} satisfies the two identities
\begin{equation}\label{eq:center-conj}
 TX_0=(X-\tfrac N2\Id)T,
 \qquad
 TY_0=Y^{(n+1)}T.
\end{equation}
For completeness, the $(r,p)$ entries of the first identity reduce to
\begin{align*}
 &-(N-p)T_{r,p}+(N-p+1)T_{r,p-1}\\
 &\qquad=-\frac r4T_{r-1,p}-(N-r)T_{r+1,p}-\frac N2T_{r,p},
\end{align*}
and those of the second to
\begin{align*}
 &-\frac{p(4n+1-p)}nT_{r,p}
 +\frac{(p+1)(2n-p)}nT_{r,p+1}\\
 &\qquad=-\frac r n(4n+1-r)T_{r,p}
 -\frac{r(r-1)}{4n}T_{r-2,p}.
\end{align*}
Both follow immediately after substituting~\eqref{eq:Tdef} and the standard binomial ratios; terms with indices outside their ranges are zero.

Now multiply~\eqref{eq:Hmatrix} by the centering gauge $\ee^{Nx/2}$ and by $T$.  Using~\eqref{eq:center-conj}, and then setting $x=-\ii\lambda$, gives~\eqref{eq:Kplus}.
\end{proof}

\section{A parameter-reflection intertwiner}\label{sec:intertwiner}

The bridge between~\eqref{eq:Msystem} and~\eqref{eq:Kplus} is an explicit triangular differential transform.  For $m\ge0$, write $(\eta)^{\uparrow m}=\eta(\eta+1)\cdots(\eta+m-1)$.

\begin{lemma}[Reflection $\eta\mapsto-\eta$]\label{lem:duality}
Suppose $K$ solves
\[
 \Dop K=(Y^{(\eta)}-\ii\lambda X)K.
\]
On a punctured domain with a fixed branch of $\lambda^{2\eta}$, define
\begin{equation}\label{eq:intertwiner}
 (\mathcal T_\eta K)_r
 :=\lambda^{2\eta}\sum_{j=0}^{r}\binom rj
 \left(-\frac{\ii}{n\lambda}\right)^{r-j}
 (\eta)^{\uparrow(r-j)}K_j,
 \qquad 0\le r\le N.
\end{equation}
Then $U:=\mathcal T_\eta K$ satisfies
\begin{equation}\label{eq:dual-system}
 \Dop U=(Y^{(-\eta)}-\ii\lambda X)U.
\end{equation}
\end{lemma}

\begin{proof}
Put
\[
 a_{rj}:=\binom rj\left(-\frac{\ii}{n\lambda}\right)^{r-j}
 (\eta)^{\uparrow(r-j)},
 \qquad m:=r-j.
\]
Then $\Dop a_{rj}=-ma_{rj}$.  Write
\[
 d_r(\eta):=-\frac r n(\eta+3n-r),
 \qquad
 c_r:=-\frac{r(r-1)}{4n}.
\]
The $r$-th component of the $\eta$-system is
\begin{equation}\label{eq:eta-component}
 \Dop K_r=d_r(\eta)K_r+c_rK_{r-2}
 +\ii\lambda\frac r4K_{r-1}+\ii\lambda(N-r)K_{r+1}.
\end{equation}
After differentiating~\eqref{eq:intertwiner} and inserting~\eqref{eq:eta-component}, equality with the $(-\eta)$-system reduces coefficientwise to
\begin{equation}\label{eq:int-id1}
 a_{r,s+2}c_{s+2}+\ii\lambda\frac{s+1}{4}a_{r,s+1}
 =c_ra_{r-2,s}+\ii\lambda\frac r4a_{r-1,s}
\end{equation}
and
\begin{equation}\label{eq:int-id2}
 \bigl(2\eta-m+d_s(\eta)\bigr)a_{r,s}
 +\ii\lambda(N-s+1)a_{r,s-1}
 =d_r(-\eta)a_{r,s}+\ii\lambda(N-r)a_{r+1,s}.
\end{equation}
The first identity follows from binomial ratios.  For the second,
\[
 \frac{a_{r,s-1}}{a_{r,s}}
 =\frac{s}{m+1}\left(-\frac{\ii}{n\lambda}\right)(\eta+m),
 \qquad
 \frac{a_{r+1,s}}{a_{r,s}}
 =\frac{r+1}{m+1}\left(-\frac{\ii}{n\lambda}\right)(\eta+m).
\]
Since $r=s+m$ and $N=2n$,~\eqref{eq:int-id2} reduces to
\[
 2\eta-m+d_s(\eta)-d_r(-\eta)
 +\frac{(r+s-N)(\eta+m)}n=0,
\]
which is an identity.  This proves~\eqref{eq:dual-system}.
\end{proof}

Apply Lemma~\ref{lem:duality} to the Selberg solution~\eqref{eq:Kplus} with $\eta=n+1$.  Since $2\eta=N+2$ is an integer, the negative powers in~\eqref{eq:intertwiner} are canceled by the prefactor.  Thus
\begin{equation}\label{eq:Udef}
 U:=\mathcal T_{n+1}K
\end{equation}
is entire and satisfies the same system as $M$:
\begin{equation}\label{eq:Usystem}
 \Dop U=(Y^{(-(n+1))}-\ii\lambda X)U.
\end{equation}
Its first component is
\begin{equation}\label{eq:Ufirst}
 U_0(\lambda)=\lambda^{N+2}K_0(\lambda)
 =\lambda^{N+2}G_n(\lambda).
\end{equation}
Since $K_0(0)=1$, the lowest power of $\lambda$ occurs only in the last component, giving
\begin{equation}\label{eq:Uasymp}
 U(\lambda)=c_I\lambda^2e_N+O(\lambda^3),
 \qquad
 c_I=\left(-\frac\ii n\right)^N(n+1)^{\uparrow N}
 =(-1)^n\frac{(3n)!}{n^{2n}n!}.
\end{equation}

\section{Frobenius matching and direct equivalence}\label{sec:frob}

\begin{lemma}[Uniqueness of the exponent-$2$ branch]\label{lem:frob}
For
\begin{equation}\label{eq:minus-system}
 \Dop V=(Y^{(-(n+1))}-\ii\lambda X)V,
\end{equation}
there is, up to an overall scalar, at most one analytic solution of the form
\[
 V(\lambda)=\lambda^2(v_0+O(\lambda)),
 \qquad v_0\ne0.
\]
Its leading vector is proportional to $e_N$.
\end{lemma}

\begin{proof}
The diagonal entries of the lower-triangular matrix $Y^{(-(n+1))}$ are
\begin{equation}\label{eq:minus-diag}
 \delta_r=\frac{r(r+1-N)}n,
 \qquad 0\le r\le N.
\end{equation}
The equation $\delta_r=2$ factors as
\[
 (r-N)(r+1)=0,
\]
so $2$ is a simple eigenvalue, occurring only at $r=N$; its eigenspace is spanned by $e_N$.

Write
\[
 V(\lambda)=\lambda^2\sum_{m\ge0}v_m\lambda^m.
\]
The coefficient recursion is
\begin{equation}\label{eq:frob-rec}
 \bigl((m+2)\Id-Y^{(-(n+1))}\bigr)v_m
 =-\ii Xv_{m-1},
 \qquad v_{-1}:=0.
\end{equation}
For $m\ge1$, the matrix on the left is invertible because every diagonal entry of $Y^{(-(n+1))}$ other than the terminal value $2$ is nonpositive.  Hence $v_0$ determines all subsequent coefficients.
\end{proof}

\begin{proof}[Proof of Theorem~\ref{thm:intro-equiv}]
The vectors $M$ and $U$ satisfy the same system~\eqref{eq:minus-system}.  Their leading terms are~\eqref{eq:Masymp} and~\eqref{eq:Uasymp}.  Lemma~\ref{lem:frob} therefore gives
\begin{equation}\label{eq:MUrelation}
 M(\lambda)=\frac{c_Q}{c_I}U(\lambda).
\end{equation}
The quotient is
\begin{equation}\label{eq:normalization-ratio}
 \frac{c_Q}{c_I}
 =\frac{(-1)^n/\binom{2n}{n}}
 {(-1)^n(3n)!/(n^{2n}n!)}
 =\frac{n^{2n}(n!)^3}{(2n)!(3n)!}
 =D_n.
\end{equation}
Taking first components in~\eqref{eq:MUrelation} and using~\eqref{eq:Mfirst} and~\eqref{eq:Ufirst} yields
\[
 \lambda^2F_n(\lambda)=D_n\lambda^{N+2}G_n(\lambda).
\]
Hence
\[
 F_n(\lambda)=D_n\lambda^NG_n(\lambda),
\]
first for $\lambda\ne0$ and then everywhere by analyticity.  Combining this with~\eqref{eq:Fq} and the Qu--Valk\'o Taylor representation~\eqref{eq:qv-series} proves~\eqref{eq:intro-equivalence}.
\end{proof}

\section{Coefficientwise identification}\label{sec:coeff}

Let $(U_1,\ldots,U_N)$ have probability density $W_n/\mathcal S_n$ on $[0,1]^N$ and define the centered trace
\begin{equation}\label{eq:Zn}
 Z_n:=\sum_{j=1}^{N}U_j-n.
\end{equation}
The symmetry $u_j\mapsto1-u_j$ gives
\begin{equation}\label{eq:Gmoment}
 G_n(\lambda)=\E\ee^{\ii\lambda Z_n}
 =\sum_{r\ge0}\frac{(-1)^r\E[Z_n^{2r}]}{(2r)!}\lambda^{2r}.
\end{equation}
Comparison with Theorem~\ref{thm:intro-equiv} yields the following direct dictionary between the matrix recursion and Selberg moments.

\begin{corollary}\label{cor:coeff}
For $1\le j<n$,
\[
 v_n^Ts_{2j}=0.
\]
For every $r\ge0$,
\begin{equation}\label{eq:coeff-id}
 2v_n^Ts_{2(n+r)}
 =D_n\frac{(-1)^r}{(2r)!}\E[Z_n^{2r}].
\end{equation}
In particular,
\[
 2v_n^Ts_{2n}=D_n,
\]
recovering the leading small-$\lambda$ coefficient.
\end{corollary}

\section{Discussion}

The two results use different structural features of the Qu--Valk\'o representation.  Joint continuity is obtained before any Fourier-mode estimates are used: stochastic monotonicity of the diffusion turns separate parameter continuity into joint weak continuity of its terminal law.  The Palm formula then transfers this regularity to the pair correlation function.

For even beta, the equivalence proof reduces two apparently different descriptions to one finite-dimensional system.  Three explicit operations are involved:
\begin{enumerate}[label=(\roman*)]
\item a Krawtchouk transform of the Qu--Valk\'o Fourier-mode ODE;
\item centering of the Jacobi--Selberg trace system;
\item the triangular parameter-reflection transform~\eqref{eq:intertwiner}.
\end{enumerate}
The normalization in Forrester's formula is then forced by the leading Frobenius vectors rather than inserted separately.  Corollary~\ref{cor:coeff} shows that the agreement is coefficientwise: the Qu--Valk\'o matrix recursion encodes the centered even moments of a finite Jacobi--Selberg trace.

\paragraph{Relation to fusion asymptotics.}
The coefficientwise identity also gives a direct comparison with the author's short-distance results.  Writing $\mathfrak C_\beta$ for the leading pair-fusion amplitude, the second-order expansion in~\cite{FangSecond} is
\[
 \frac{\rhoS_\beta(0,\lambda)}{\mathfrak C_\beta|\lambda|^\beta}
 =1-\frac{\beta^2}{40(2\beta-1)}\lambda^2+o(\lambda^2),
 \qquad \beta>\tfrac12.
\]
At $\beta=2n$, the leading amplitude is $\mathfrak C_{2n}=D_n/(4\pi^2)$, and this normalized coefficient becomes $-n^2/[10(4n-1)]$.  By Theorem~\ref{thm:intro-equiv} and~\eqref{eq:Gmoment}, the same coefficient is $-\E[Z_n^2]/2$.  Thus comparison with~\cite{FangSecond} yields
\[
 \E[Z_n^2]=\frac{n^2}{5(4n-1)}.
\]
This is a consistency relation between the Selberg trace moments and the fusion expansion, rather than an input to the equivalence proof.

For $0<\beta<\tfrac12$, the normalized first correction has order $|\lambda|^{1+2\beta}$, while at $\beta=\tfrac12$ it has order $\lambda^2\log(1/|\lambda|)$; see~\cite{FangCritical}.  These fixed-parameter asymptotics complement the joint continuity proved in Theorem~\ref{thm:intro-cont}.  Our continuity proof is independent of them and does not require an interchange of the collision limit with the inverse-temperature limit.


\begin{thebibliography}{99}

\bibitem{Aomoto}
K.~Aomoto,
\emph{Jacobi polynomials associated with Selberg's integral},
SIAM J. Math. Anal. \textbf{18} (1987), 545--549.

\bibitem{FangSecond}
W.~Fang,
\emph{Second-order fusion asymptotics for $\Sine_\beta$ correlation functions},
preprint (2026), \href{https://arxiv.org/abs/2608.23742}{arXiv:2608.23742}.

\bibitem{FangCritical}
W.~Fang,
\emph{Critical and subcritical fusion asymptotics for $\Sine_\beta$ correlation functions},
preprint (2026), \href{https://arxiv.org/abs/2609.07239}{arXiv:2609.07239}.

\bibitem{Forrester1992}
P.~J. Forrester,
\emph{Selberg correlation integrals and the $1/r^2$ quantum many-body system},
Nuclear Phys. B \textbf{388} (1992), 671--699.

\bibitem{Forrester1994}
P.~J. Forrester,
\emph{Addendum to ``Selberg correlation integrals and the $1/r^2$ quantum many-body system''},
Nuclear Phys. B \textbf{416} (1994), 377--385.

\bibitem{ForresterBook}
P.~J. Forrester,
\emph{Log-Gases and Random Matrices},
London Mathematical Society Monographs Series 34,
Princeton University Press, 2010.

\bibitem{ForresterKumar}
P.~J. Forrester and S.~Kumar,
\emph{Differential recurrences for the distribution of the trace of the $\beta$-Jacobi ensemble},
Physica D \textbf{434} (2022), 133220.

\bibitem{QuValko}
Y.~Qu and B.~Valk\'o,
\emph{On the pair correlation function of the $\Sine_\beta$ process},
arXiv:2509.15446, 2025.

\end{thebibliography}
\end{document}